\documentclass[12pt,psamsfonts,reqno]{amsart}

\usepackage{amssymb,amsfonts}
\usepackage[all,arc]{xy}
\usepackage{enumerate}
\usepackage{mathrsfs}

\usepackage[top=1in, bottom=1.25in, left=1.25in, right=1.25in]{geometry}
\usepackage{amsfonts}
\usepackage{amsthm}
\usepackage{amsmath}
\usepackage{parskip}
\usepackage{mathrsfs}
\usepackage{graphicx}
\usepackage{amssymb}
\usepackage{xtab}
\usepackage{mathtools}

\usepackage{hyperref}
\usepackage{cleveref}
\usepackage{enumerate}
\usepackage{tikz} 
\usepackage{float}
\usetikzlibrary{positioning}
\usetikzlibrary{shapes.geometric,arrows}

\newtheorem{thm}{Theorem}[section]

\newtheorem{prop}[thm]{Proposition}
\newtheorem{lem}[thm]{Lemma}

\theoremstyle{definition}
\newtheorem{defn}[thm]{Definition}

\newtheorem{exmp}[thm]{Example}

\theoremstyle{remark}
\newtheorem{rem}[thm]{Remark}


\numberwithin{equation}{section}

\newcommand{\q}{\mathbb{Q}}

\DeclareMathOperator{\Norm}{Norm}

\newcommand*{\myproofname}{Proof of \Cref{thm: bPefectPowers}}
\newenvironment{myproof}[1][\myproofname]{\begin{proof}[#1]}{\end{proof}}

\title{On the Finiteness of Perfect Powers in Elliptic Divisibility Sequences}
\author[A. Alfaraj]{Abdulmuhsin Alfaraj}
\address{Abdulmuhsin Alfaraj\\
	Department of Mathematical Sciences \\
	University of Bath \\
	Claverton Down \\
	Bath \\
	BA2 7AY \\
	UK.}
\urladdr{}

\begin{document}

\begin{abstract}
	We prove that there are finitely many perfect powers in elliptic divisibility sequences generated by a non-integral point on elliptic curves of the from $y^2=x(x^2+b)$, where $b$ is any positive integer. We achieve this by using the modularity of elliptic curves over real quadratic number fields.
\end{abstract}

\maketitle
\tableofcontents

\setcounter{equation}{0}

\section{Introduction}
We start by explaining what we mean by an elliptic divisibility sequence.
\begin{defn}\label{def:EDS}
Let $E/\q$ be an elliptic curve given in Weierstrass form. Let $P\in E(\q)$ be a non-torsion point. For all $m\in \mathbb{N}$, we can write $$mP=\left(A_m/B_m^2, C_m/B_m^3 \right),$$ where $A_m, B_m, C_m$ are integers, $B_m\neq 0$, and $\gcd(A_m,B_m)=\gcd(C_m,B_m)=1$. We say that $\{B_m\}$ is the elliptic divisibility sequence generated by $P$.
\end{defn}

In this paper we prove the finiteness of perfect powers in elliptic divisibility sequences generated by a non-integral point in a new class of elliptic curves.
\begin{thm}\label{thm: bPefectPowers}
	Let $\{B_m\}$ be an elliptic divisibility sequence generated by a non-integral point on any  elliptic curve over $\q$ of the form $$E:\quad y^2=x(x^2+b),$$ where $b$ is a positive integer.
	Then there are finitely many perfect powers in $\{B_m\}$.
\end{thm}

We prove this using the so-called modular approach, which is inspired by the proof of Fermat's last theorem, along with utilizing the properties of elliptic divisibility sequences. The modular method was first used to study integer solutions to special classes of Diophantine equations, and it mainly relies on the modularity theorem and Ribet's level-lowering theorem. Recently, Freitas, Le Hung, and Siksek \cite{FLS} proved that elliptic curves over real quadratic number fields are modular. Moreover, Fujiwara  \cite{Fuji}, Jarvis \cite{Jarvis} and Rajaei \cite{Rajaei} proved level lowering theorems for Hilbert
eigenforms which can be considered as generalisations of Ribet's level
lowering theorem. This allowed us to apply this modular method when working with equations over real quadratic fields, which was a main part of our proof.

The study of the finiteness of perfect powers in sequences first started with the Fibonacci sequence and the Lucas sequence. By utilizing modular methods along with other techniques, Bugeaud, Mignotte and Siksek \cite{FibLuc} proved that the only perfect powers in the Fibonacci sequence and the Lucas sequence are $\{0, 1, 8,144\}$ and  $\{1,4\}$, respectively. This motivated the study of perfect powers in other special sequences, such as elliptic divisibility sequences. Everest and King \cite{EverestPP} showed that, when assuming certain conditions on the generating point, the corresponding elliptic divisibility sequence contains only a finite number of elements that are powers of primes.  In \cite{Everest4}, Everest, Reynolds, and Stevens showed that the number of perfect powers of a fixed exponent in an elliptic divisibility sequence is finite by invoking Faltings' theorem.  If the first term of the sequence is divisible by $2$ or $3$, or if the sequence arises from a Mordell curve where the first term is greater than $1$, Reynolds \cite{Jrey2} has shown that there are only finitely many terms that are perfect powers in both of these cases by using the modular approach. \\

\textbf{Outline of the paper.} We start by collecting some relevant results on elliptic curves in section \ref{sec:ec}. In section \ref{sec:EDS}, we state some properties of elliptic divisibility sequences. In section \ref{sec:Mod} we briefly cover results on the modularity of elliptic curves over quadratic number fields along with a level-lowering theorem. Finally, we prove \Cref{thm: bPefectPowers} in section \ref{sec:proof}. \\

\textbf{Acknowledgements.} I would like to thank Samir Siksek for sharing his knowledge and for his guidance throughout my master's project at the University of Warwick, from which this paper was produced. I would also like to thank Daniel Loughran for his guidance and comments.

\section{Elliptic Curves}\label{sec:ec}

In this section we collect some relevant results on elliptic curves over number fields. The primary aim is to provide sufficient conditions to check if a given equation of an elliptic curve has an equation that is minimal at a prime ideal, and to identify its type of reduction modulo a prime ideal.

Let $E$ be an elliptic curve over some field $K$. We say that $E/K$ is given in Weierstrass form if $$ E\>:\> y^2+a_1xy+a_3y=x^3+a_2x^2+a_4x+a_6, $$
where $a_1,a_2,a_3,a_4,a_6\in K$. We define the quantities $\Delta$ and $c_4$ as in III.1 in \cite{Silverman1}.

Let $K$ be a number field and $\mathcal{O}_K$ be its ring of integers. Let $v_{\mathfrak{p}}$ be the discrete valuation of $K$ with respect to the prime ideal $\mathfrak{p}$. We write $K_{\mathfrak{p}}$ for the completion of $K$ at $v_{\mathfrak{p}}$, $\mathcal{O}_{\mathfrak{p}}$ for its ring of integers, ${\mathfrak{p}}$ for the maximal ideal of $\mathcal{O}_{\mathfrak{p}}$, and $k_{\mathfrak{p}}=\mathcal{O}_{\mathfrak{p}}/{\mathfrak{p}}$, the residue field of $\mathcal{O}_{\mathfrak{p}}$. We say that a Weierstrass equation for $E/K$ is minimal at $\mathfrak{p}$ if $v_{\mathfrak{p}}(\Delta)$ is minimal subject to the condition that $a_1,a_2,a_3,a_4,a_6\in \mathcal{O}_{{\mathfrak{p}}}$.

\begin{rem}\label{rem:minimalcriteria}
	Suppose we are given a Weierstrass equation of $E/K$ with coefficients in $\mathcal{O}_{\mathfrak{p}}$. If $v_{\mathfrak{p}}(\Delta)<12$ or $v_{\mathfrak{p}}(c_4)<4$, then the equation is minimal at $\mathfrak{p}$ (Remark VII.1.1 in \cite{Silverman1}).
\end{rem}

\begin{prop}[VII.5.1 \cite{Silverman1}] \label{prop:reduction-criteria}
	Let $E/K$ be an elliptic curve given by a minimal Weierstrass equation for $E$ at some $\mathfrak{p}$. Then
	\begin{enumerate}
		\item $E$ has good reduction at $\mathfrak{p}$ if and only if $v_{ \mathfrak{p}}(\Delta)=0$;
		\item  $E$ has multiplicative reduction at $\mathfrak{p}$ if and only if $v_{ \mathfrak{p}}(\Delta)>0$ and $v_{ \mathfrak{p}}(c_4)=0$;
		\item and $E$ has additive reduction at $\mathfrak{p}$ if and only if $v_{ \mathfrak{p}}(\Delta)>0$ and $v_{ \mathfrak{p}}(c_4)>0$.
	\end{enumerate}
\end{prop}

We say that $E$ has semistable reduction at $\mathfrak{p}$ if it has either good or multiplicative reduction at $\mathfrak{p}$. We warn the reader that this terminology is used exclusively for multiplicative reduction in \cite{Silverman1}.

\begin{defn}
	We define the conductor of $E/K$ to be 
	$$\mathcal{N}=\prod_{\mathfrak{p} \text{ prime ideal} } \mathfrak{p}^{f_{\mathfrak{p}}(E)},$$
	where for all primes $\mathfrak{p}$ not dividing $2$ or $3$,
	\[	f_\mathfrak{p}(E) =
	\begin{cases}
		0 & \text{if $E$ has good reduction at $\mathfrak{p}$}, \\
		1 & \text{if $E$ has multiplicative reduction at $\mathfrak{p}$}, \\
		2 & \text{if $E$ has additive reduction at $\mathfrak{p}$};
	\end{cases}
	\]
	and for $\mathfrak{p}$ dividing $2$ or $3$,  $f_\mathfrak{p}(E)$ is defined similarly, except for when $E$ has additive reduction; in which case, $f_\mathfrak{p}(E)$ may have values greater than $2$. 
\end{defn}

\begin{rem}\label{rem:boundNpowers23}
	By Theorem IV.10.4 in \cite{SilvermanAdv}, for all primes $\mathfrak{p}$,  $f_\mathfrak{p}(E)$ has a bound depending only on the number field $K$. 
\end{rem}

\vspace{5 pt}

\section{Elliptic Divisibility Sequences}\label{sec:EDS}

In this section we state some properties of elliptic divisibility sequences that we will be using in the proof of \Cref{thm: bPefectPowers}.

First, we note that our definition of an elliptic divisibility sequence is the one used by Silverman in \cite{IngramSilverman}, and by Everest in \cite{EverestPP}. The term elliptic divisibility sequence was also used by Ward in \cite{Ward} to define a closely related class of sequences satisfying certain recurrence relations. In this paper, we work with \Cref{def:EDS} of an elliptic divisibility sequence. 

We say that an element of a sequence of integers has a primitive divisor if it is divisible by a prime $p$ that does not divide any element preceding it. 

\begin{exmp}
	Consider the elliptic curve $$E: \qquad y^2=x(x^2+5)$$ over $\q$.
	The Mordell-Weil group of $E(\q)$ has rank $1$, with generator $P=(20,90)$, and torsion subgroup equal to the idenitity and the $2$-torsion point $(0,0)$. Consider the elliptic divisibility sequence $\{B_m\}$ generated by $mP$. We compute
	\begin{gather*} 2P=\left( \frac{6241}{36^2}, \frac{543599}{36^3} \right), \quad 3P=\left( \frac{700217780}{19679^2}, \frac{29468421431730}{19679^3} \right), \text{ and}\\
		\quad 4P=\left( \frac{933424765104001}{39139128^2}, \frac{108467911710220197291841}{39139128^3} \right).
	\end{gather*}   
	Thus, the first $4$ elements of the sequence are $B_1=1$, $B_2=36$, $B_3=19679$ , and $B_4= 39139128$. Note that the primitive divisors of $B_2$ are $\{2,3\}$, of $B_3$ are $\{11,1789\}$, and of $B_4$ are $\{7,79,983\}$.
\end{exmp}

\begin{thm}[Silverman \cite{SilvermanPrimitive}]\label{thm:silverman}
	Let $\{B_m\}$ be an elliptic divisibility sequence. Then, the number of elements in $\{B_m\}$ not having a primitive divisor is finite.
\end{thm}

Elliptic divisibility sequences satisfy certain divisibility properties as the name suggests. The following proposition shows that they are strong divisibility sequences and that we have some form of control on the valuations of the terms.
\begin{prop}[Lemma IV.3.1 \& Corollary IV.4.5  \cite{String}]\label{prop:bmp1}
	Let $\{B_m\}$ be an elliptic divisibility sequence generated by a non-torsion point on some elliptic curve over $\q$ given by an integral Weierstrass equation. 
	\begin{enumerate}[(i)]
		\item Let $p$ be an odd prime, and $n,\> m\in \mathbb{N}$. If $v_p(B_n)>0$, then \begin{equation*}
			v_p(B_{nm})=v_p(B_n)+v_p(m).
		\end{equation*}
		Suppose that $p=2$ and $2\mid a_1$, the coefficient in the Weiestrass equation. If $v_2(B_n)>0$, then \begin{equation*}
			v_2(B_{nm})=v_2(B_n)+v_2(m).
		\end{equation*}
		\item $\{B_m\}$ is a strong divisibility sequence, i.e., for all $n,\> m\in \mathbb{N}$, $$\gcd(B_m,B_n)=B_{\gcd(m,n)}.$$
	\end{enumerate}
\end{prop}

Note that for the elliptic curve we are considering in \Cref{thm: bPefectPowers}, we have that $a_1=0$. Thus, the condition $2\mid a_1$ is satisfied. 

As a consequence of \Cref{thm:silverman} and \Cref{prop:bmp1}, we have the following lemma which will be a crucial step in proving \Cref{thm: bPefectPowers}. It says that whenever the first element of the sequence has a prime divisor, and given a finite set of primes, any element of the sequence that is an $\ell$th power for a prime $\ell$ large enough will have a prime divisor outside this finite set.

\begin{lem}\label{lem:primdivisorconstruct}
	Let $\{B_m\}$ be an elliptic divisibility sequence generated by a non-torsion point on some elliptic curve over $\q$ given by an integral Weierstrass equation, where $2\mid a_1$. Suppose that $B_1$ is divisible by some prime number $q$. Let $T$ be a finite set of prime numbers. There exists a positive integer $k$ and a prime $p\notin T$ such that if $B_m$ is an $\ell$th power for some prime $\ell>k$, then $p\>|\>B_m$.
\end{lem}
\begin{proof}
	By \Cref{thm:silverman}, all but finitely many elements of $\{ B_m\}$ have primitive divisors. Since $T$ is a finite set of primes, we can always choose a large enough positive integer $k$ such that the element $B_{q^{k-v_q(B_1)}}$ has a primitive divisor $p$ not contained in $T$. 
	
	Let $B_m$ is an $\ell$th power for some prime $\ell>k$. Since $q\>|\>B_1$, \Cref{prop:bmp1} implies that $$\ell\leq v_q(B_m)=v_q(B_1)+v_q(m) \> \implies \> v_q(m)\geq k-v_q(B_1),$$ that is, $ q^{k-v_q(B_1)}\>|\> m$.
	Thus, by \Cref{prop:bmp1}, we have
	$$p\>|\>B_{q^{k-v_q(B_1)}}=B_{\gcd(q^{k-v_q(B_1)},m)}=\gcd(B_{q^{k-v_q(B_1)}},B_m).$$
	Therefore, $p\>|\>B_m$.
\end{proof}
\vspace{5 pt}

\section{Modularity and Level-Lowering}\label{sec:Mod}

Let $K$ be a totally real number field, and denote the absolute Galois group of $K$ by $G_K = Gal(\bar{\q}/K)$. Let $ E $ be an elliptic curve over $ K $, and write $\rho_{ E,p}$ for the Galois representation arising from the action of $G_K$ on the $p$-adic Tate module $T_p(E)$. We say that $E$ is modular if there exists a Hilbert cuspidal eigenform $\mathfrak{f}$ of parallel weight $2$ with rational Hecke eigenvalues such that the $L$-function of $\mathfrak{f}$ is equal to the Hasse-Weil $L$-function of $E$. Another way to express this is that there exists an isomorphism of compatible systems of Galois representations  $ \rho_{ E,p} \cong \rho_{ \mathfrak{f},p} $
where $\rho_{ \mathfrak{f},p}$ is the Galois representation associated to $\mathfrak{f}$ by Eichler and Shimura \cite{DiamondShurman} for $K=\q$, and by Carayol \cite{Cay1} \cite{Cay2}, Blasius and Rogawski \cite{BlasRog}, Wiles \cite{WilesReps} and Taylor \cite{TayGalRepsH} for any totally real number field $K$.

\begin{thm}[Freitas, Le Hung, \& Siksek \cite{FLS}]\label{thm:RQmodulra}
	Let $E$ be an elliptic curve over a real quadratic number field. Then $E$ is modular.
\end{thm}

Denote by $$\bar{\rho}_{E,p}:G_K\rightarrow \text{Aut}(E[p])\cong \text{GL}_2(\mathbb{F}_p)$$ the representation giving the action of $G_K$ on $E[p]$, the $p$-torsion of $E$. Let $\mathfrak{f}$ be Hilbert eigenform $\mathfrak{f}$ over $K$. We denote by $\q_\mathfrak{f}$ the number field generated by the eigenvalues $a_{\mathfrak{q}}(\mathfrak{f})$  at all prime ideals $\mathfrak{q}$ of $\mathcal{O}_K$.

The following theorem, stated in \cite{FS}, is a useful summary of level lowering results
of Fujiwara  \cite{Fuji}, Jarvis \cite{Jarvis} and Rajaei \cite{Rajaei} in the context of modular elliptic
curves.

\begin{thm}\label{thm:Level-Low}
	Let $K$ be a totally real number field, and $E/K$ an elliptic
	curve of conductor $\mathcal{N}$. Let $\ell$ be a rational prime. For a prime ideal $\mathfrak{q}$ of $K$ denote
	by $\Delta_{\mathfrak{q}}$ the discriminant of a local minimal model for E at $\mathfrak{q}$. Let  $$\mathcal{M}_\ell:=\prod_{\mathfrak{p}||\mathcal{N}, \> \ell \mid v_{\mathfrak{p}} (\Delta_{\mathfrak{p}})} \mathfrak{p}, \>\>\>\> \mathcal{N}_\ell:=\frac{\mathcal{N}}{\mathcal{M}_\ell}.$$ 
	
	Suppose the following
	\begin{enumerate}[(i)]
		\item $\ell\geq 5$, the ramification index $e(\mathfrak{q}/\ell)<\ell-1$ for all $\mathfrak{q}\mid \ell$, and $\q(\zeta_l)^+\not\subset K$;
		\item $E$ is modular;
		\item $\bar{\rho}_{E,\ell}$ is irreducible;
		\item $E$ is semistable at all $\mathfrak{q}\mid \ell$; 
		\item $\ell \mid v_{\mathfrak{q}} (\Delta_{\mathfrak{q}})$ for all $\mathfrak{q}\mid \ell$.
	\end{enumerate}
	
	Then, there is a Hilbert eigenform $\mathfrak{f}$ of parallel weight $2$ that is new at level $\mathcal{N}_\ell$ and some prime ideal $\lambda$ of the ring of integers of $\q_\mathfrak{f}$ such that $\lambda\mid \ell$ and $\bar{\rho}_{E,\ell}\sim \bar{\rho}_{\mathfrak{f},\lambda}$.
\end{thm}

\begin{rem}\label{rem:arises_consequence}
	Following the notation of \Cref{thm:Level-Low}, suppose that $\bar{\rho}_{E,\ell}\sim \bar{\rho}_{\mathfrak{f},\lambda}$. Let $\mathfrak{p}$ be a non-zero prime ideal in $K$. By comparing the traces of the images of Frobenius at $\mathfrak{p}$ in $\bar{\rho}_{F,\ell}$ and $\bar{\rho}_{\mathfrak{f},\lambda}$ we have the following. 
	\begin{enumerate}[(i)]
		\item If $\mathfrak{p}\nmid \ell\mathcal{N}$, then $a_{\mathfrak{p}}(E) \equiv a_{\mathfrak{p}}(\mathfrak{f}) \mod \lambda.$ 
		\item If $\mathfrak{p} \nmid \ell\mathcal{N}_\ell$ and $\mathfrak{p} \>||\> \mathcal{N}$, then $\pm(\Norm_{{K}/\q}(\mathfrak{p})+1)\equiv a_{\mathfrak{p}}(\mathfrak{f}) \mod \lambda.$
	\end{enumerate}
	
\end{rem}

We will be using the following result which shows that, under certain assumptions, $(iii)$ in \Cref{thm:Level-Low} is satisfied whenever $\ell$ is greater than some bound depending only the number field $K$.

\begin{thm}[Freitas \& Siksek \cite{FSirrModp}]\label{thm:RepIrr}
	Let $K$ be a Galois totally real field. There is an effective constant
	$C_K$, depending only on $K$, such that the following holds. If $\ell > C_K$ is prime, and $E$
	is an elliptic curve over $K$ which is semistable at all $\mathfrak{q} \mid \ell$, then $\bar{\rho}_{E,\ell}$ is irreducible.
\end{thm}

\vspace{5 pt}

\section{Proof of \Cref{thm: bPefectPowers}}\label{sec:proof}

To prove \Cref{thm: bPefectPowers}, the crucial step is to bound the exponent of any perfect power. To do that, we associate a specific Frey curve for any point $mP$ with $B_m$ a large enough $\ell$th power, and by using modularity and level lowering, we obtain a bound that only depends on the point $P$ and the positive integer $b$. In fact, the proof provides an effective method to compute this bound explicitly.

We start by proving a lemma that will allow us to associate an appropriate Frey curve to an equation obtained from substituting $mP$ in the equation of $E$. 

\begin{lem}\label{lem:FreyCurveConstruction}
	Let $\ell$ be a prime, and let $a,d,u,v,$ and $w$ be non-zero integers.
	Suppose that $a,d$ are positive, $a$ is square-free, $\gcd(u,v)\mid ad$, and that
	\begin{equation}\label{eqn:sol-to-Frey}
		v^2-au^4=dw^{4\ell}.
	\end{equation}
	Then we can associate to \eqref{eqn:sol-to-Frey} the Frey curve  \begin{equation}F:\>\> Y^2=X(X^2+4u\sqrt{a}X+2\sqrt{a}(v+u^2\sqrt{a})),
	\end{equation}
    over the number field $K=\q(\sqrt{a})$, where $\Delta_F\neq 0$. Moreover, for all non-zero prime ideals $\mathfrak{p}$ of $\mathcal{O}_K$ such that $\mathfrak{p}\nmid 2ad$, the curve $F$ is minimal at $\mathfrak{p}$, has semistable reduction at $\mathfrak{p}$, and $\ell\mid v_{\mathfrak{p}}(\Delta_{F})$.
\end{lem}

\begin{proof}
	Factoring over $\mathcal{O}_K$, we have 
	\begin{equation}\label{eqn:factor} 
		(v+ u^2\sqrt{a})(v- u^2\sqrt{a})=dw^{4\ell}. \end{equation}
	If $\mathfrak{q}$ is a non-zero a prime ideal of $\mathcal{O}_K$ that divides both $(v+ u^2\sqrt{a})$ and $(v- u^2\sqrt{a})$, then $\mathfrak{q}\mid 2ad$, since $\gcd(u,v)\mid ad$.
	Thus, by the unique factorization of ideals into products of prime ideals, we obtain
	\begin{equation}\label{eqn:primeidealLpower} (v+u^2\sqrt{a})\mathcal{O}_K=\mathfrak{b}_1\mathfrak{a}_1^\ell \>\>\> \text{ and } \>\>\> (v-u^2\sqrt{a})\mathcal{O}_K=\mathfrak{b}_2\mathfrak{a}_2^\ell,
	\end{equation}
	where $\mathfrak{a}_1, \mathfrak{a}_2,\mathfrak{b}_1, \mathfrak{b}_2$ are ideals in $\mathcal{O}_K$, and $\mathfrak{b}_1 \mathfrak{b}_2\mid 2ad\mathcal{O}_K$. Now, notice that 
	\begin{equation}\label{ll2quadcase}(v+u^2\sqrt{a})-(v-u^2\sqrt{a})= 2u^2\sqrt{a}.\end{equation}
	We can associate the following Frey curve to \eqref{ll2quadcase}:
	$$F:\>\> Y^2=X(X^2+4u\sqrt{a}X+2\sqrt{a}(v+u^2\sqrt{a})),$$
	where \begin{equation}\label{eqn:disc} \Delta_F=-2^9a\sqrt{a}(v+u^2\sqrt{a})^2(v-u^2\sqrt{a})\end{equation} and $$c_4=32\sqrt{a}(3v-5u^2\sqrt{a}).$$
	Suppose that $a\neq 1$. Since $a$ is square free, $\sqrt{a}$ is irrational; hence, $\Delta_F\neq 0$. Now if $a=1$ and $\Delta_F=0$, then $dw^{4\ell}=(v+u^2\sqrt{a})(v-u^2\sqrt{a})=0$, contradicting that both $d$ and $w$ are non-zero. Therefore,  $\Delta_F$ is non-zero.
	
	Let $\mathcal{S}$ be the set of non-zero prime ideals in $\mathcal{O}_K$ dividing $2ad$. Now, if $\mathfrak{p}$ is a prime ideal of $\mathcal{O}_K$ that divides both $\Delta_F$ and $c_4$, then $\mathfrak{p}$ will either divide  $(v+u^2\sqrt{a})$ and $(3v-5u^2\sqrt{a})$, or  $(v-u^2\sqrt{a})$ and $(3v-5u^2\sqrt{a})$. If $\mathfrak{p}$ divides $(v+u^2\sqrt{a})$ and $(3v-5u^2\sqrt{a})$, then it will divide $$3(v+u^2\sqrt{a})-(3v-5u^2\sqrt{a})=8u^2\sqrt{a}$$ and $$5(v+u^2\sqrt{a})+(3v-5u^2\sqrt{a})=8v$$ implying that $\mathfrak{p}$ divides both $2u\sqrt{a}$ and $2v$. Since $\gcd (u,v) \mid ad$, we have that $\mathfrak{p}$ divides $(2d\sqrt{a})\mathcal{O}_K$. A similar argument with the other pair also gives that $\mathfrak{p} \mid (2d\sqrt{a})\mathcal{O}_K$. Thus, if $\mathfrak{p}$ divides $\Delta_F$ and $c_4$, then $\mathfrak{p}\in \mathcal{S}$.
	
	Therefore, by \Cref{rem:minimalcriteria}, the curve $F$ is minimal at all primes  $\mathfrak{p}\notin \mathcal{S}$. Moreover, by \Cref{prop:reduction-criteria}, the curve $F$ has good or multiplicative reduction at all primes $\mathfrak{p}\notin \mathcal{S}$. Also, if  $\mathfrak{p}\notin\mathcal{S}$ and $\mathfrak{p}\mid \Delta_F$, then  \eqref{eqn:primeidealLpower} and \eqref{eqn:disc} imply that $\mathfrak{p}$ divides $\mathfrak{a}_1^\ell$ or $\mathfrak{a}_2^\ell$. Hence,  $\ell\mid v_{\mathfrak{p}}(\Delta_{F})$ for all $\mathfrak{p}\notin \mathcal{S}$.	
\end{proof}

The next lemma shows that if $\ell$ in \Cref{lem:FreyCurveConstruction} is large enough, then we can effectively apply level lowering to the Frey curve $F$. We note that the $C_K$ appearing below is the effective constant, depending only on $K$, given by \Cref{thm:RepIrr} for the Galois totally real field $K$ under consideration. Note that this makes sense as $K$ in \Cref{lem:FreyCurveConstruction} is either $\q$ or a real quadratic number field, and in both cases $K$ is a Galois totally real field.

\begin{lem}\label{lem:levellowering}
	Take the assumptions of \Cref{lem:FreyCurveConstruction}. Let $\mathcal{N}$ be the conductor of $F$. If $\ell>max\{2ad,C_K,5\}$, then there is a Hilbert eigenform $\mathfrak{f}$ of parallel weight $2$ that is new at level $\mathcal{N}_\ell$ and a prime ideal $\lambda$ of $\q_\mathfrak{f}$ such that $\lambda\mid \ell$ and $\bar{\rho}_{F,\ell}\sim \bar{\rho}_{\mathfrak{f},\lambda}$.
\end{lem}

\begin{proof}
	Recall the definitions  $$\mathcal{M}_\ell:=\prod_{\mathfrak{p}||\mathcal{N}, \> \ell|v_{\mathfrak{p}} (\Delta_{\mathfrak{p}})} \mathfrak{p}, \>\>\>\> \mathcal{N}_\ell:=\frac{\mathcal{N}}{\mathcal{M}_\ell},$$ in \Cref{thm:Level-Low}. Again, let $\mathcal{S}$ be the set of prime ideals in $\mathcal{O}_K$ dividing $2ad$. By \Cref{lem:FreyCurveConstruction}, if $\mathfrak{p}\mid \mathcal{N}$ and $\mathfrak{p}\notin\mathcal{S}$, then $\mathfrak{p}\>||\>\mathcal{N}$ and $\ell\mid v_{\mathfrak{p}}(\Delta_{F})$. Then $\mathcal{N}_\ell$ is a product of finitely many primes ideals that are contained in $\mathcal{S}$. 
	
	Since $\ell>2ad$, $\ell$ is not divisible by any prime ideal in $\mathcal{S}$. Thus, for all $\mathfrak{p}\mid \ell$, $F$ has semistable reduction at $\mathfrak{p}$ and $\ell\mid  v_{\mathfrak{p}} (\Delta_{F})$. Moreover, since $\ell> C_K$,   \Cref{thm:RepIrr} implies that  $\bar{\rho}_{F,\ell}$ is irreducible. Therefore, the following conditions are satisfied: 
	\begin{enumerate}[(i)]
		\item $\ell\geq 5$, the ramification index for all primes of $\mathcal{O}_K$ lying above $\ell$ is $\leq$ 2, and $\q(\zeta_\ell)^+\not\subset {K}$;
		\item $F$ is modular, by \Cref{thm:RQmodulra} if $a\neq 1$, and by the Modularity Theorem if $a= 1$;
		\item $\bar{\rho}_{F,\ell}$ is irreducible;
		\item $E$ is semistable at $\mathfrak{p}$ for all $\mathfrak{p} \mid \ell$;
		\item and $\ell\mid  v_{\mathfrak{p}} (\Delta_{F})$ for all $\mathfrak{p} \mid \ell$.
	\end{enumerate}
	Hence, we may apply \Cref{thm:Level-Low}.
\end{proof}

\begin{myproof}
	Write $mP=\left( A_m/B_m^2, C_m/B_m^3 \right)$ in lowest terms, for all $m\geq 1$. If $A_m B_m=0$, then from the equation of the elliptic curve we deduce that $P=(0,0)$, which is a $2$-torsion point; but $P$ is non-torsion. Thus, $A_m\neq0$ and $C_m\neq0$ for all $m\geq 1$. Suppose that $B_m=w_m^\ell$, for a prime number $\ell$ and an integer $w_m$. Since $B_1>1$ and $B_1\mid B_m$ (by \Cref{prop:bmp1}), we have $w_m>1$ for all $m$. For a fixed $\ell$, by Theorem 1.1 in \cite{Everest4}, we can only have finitely many elements of the sequence that are $\ell$th powers. Thus, to prove the theorem it suffices to show that $\ell$ is bounded.
	
	By substituting $mP$ in the equation of the elliptic curve, we get an equation of the form
	\begin{equation*} C_m^2=A_m(A_m^2+bw_m^{4\ell}), \end{equation*} 
	where $\gcd(A_m,A_m^2+bw_m^{4\ell})\>|\> b$. Therefore, we can write 
	\begin{equation*} A_m=a_m u_m^2 \>\>\> \text{ and } \>\>\> A_m^2+bw_m^{4\ell}=a_m v_m^2,\end{equation*}
	where $u_m, v_m,$ and $a_m$ are positive integers, $a_m$ is square-free, $a_m \mid b$, and $\gcd(u_m,v_m)$ divides $b$. Observe that $u_m v_m\neq0$, since $C_m\neq 0$. Hence, we get the equation
	\begin{equation}\label{fixedleqn} v_m^2-a_m u_m^4=(b/a_m)w_m^{4\ell}. \end{equation}
	
	Observe that \eqref{fixedleqn} is \eqref{eqn:sol-to-Frey} in \Cref{lem:FreyCurveConstruction} with $u=u_m,\> v=v_m, \> w=w_m,$ $a=a_m,\> d=b/a_m,$ and $2ad=2b$. Then we can associate to \eqref{fixedleqn} the Frey curve  \begin{equation*}F_m:\>\> Y^2=X(X^2+4u_m\sqrt{a_m}X+2\sqrt{a_m}(v_m+u_m^2\sqrt{a_m})),
	\end{equation*}
	over the number field $K_m=\q(\sqrt{a_m})$, where $\Delta_{F_m}\neq 0$. Moreover, for all primes $\mathfrak{p}$ of $\mathcal{O}_{K_m}$ such that $\mathfrak{p}\nmid 2b$,  $F_m$ is minimal at $\mathfrak{p}$, $F_m$ has semistable reduction at $\mathfrak{p}$, and $\ell\mid v_{\mathfrak{p}}(\Delta_{F_m})$.

	Let $\mathcal{S}$ be the set of prime ideals in $\mathcal{O}_{K_m}$ dividing $2b$. Let $T$ be the set of prime numbers lying under the prime ideals of $\mathcal{S}$. Then, by \Cref{lem:primdivisorconstruct}, there exists a positive integer $k$, depending only on $B_1$ and $b$, and a prime $p_0\notin T$ such that if $B_m$ is an $\ell$th power for some $\ell>k$, then $p_0\mid B_m$. 
	Now, for all $m\in \mathbb{N}$, $a_m\mid b$. Hence, for all $B_m$, there are only finitely many possible number fields $K_m$, which only depend on the constant $b$ appearing in the equation of $E$. Again, for any $m\in \mathbb{N}$, let $C_m$ denote the corresponding effective bound given by \Cref{thm:RepIrr}. Set $C$ to be the maximum of all the possible values of $C_m$.
	
	Suppose now that $B_m=w_m^\ell$ where $\ell>\text{max}\{k,2b, C,p_0,5\}$. Recall that $k,2b, C,$ and $p_0$ depend only on $b$ and the point $P$ generating $\{B_m\}$.  Now, $\ell>k$ implies that $p_0\mid w_m^\ell$, which in turn implies that $p_0\mid \Delta_{F_m}$. If $\mathfrak{p}\mid p_0$, where $\mathfrak{p}$ is a prime ideal in $\mathcal{O}_{K_m}$, then $\mathfrak{p}\notin\mathcal{S}$ and $\mathfrak{p} \nmid \ell$, since $p_0\notin T$ and $\ell>p_0$. Let $\mathfrak{p}$ be any prime ideal of $\mathcal{O}_{K_m}$ that divides $p_0$. So, we have that $\mathfrak{p}\> ||\> \mathcal{N}$, the conductor of $F_m$, and $\mathfrak{p}\nmid \ell\mathcal{N}_\ell$. Since $\ell>\text{max}\{2b, C,5\}$, by \Cref{lem:levellowering}, there is a Hilbert eigenform $\mathfrak{f}$ over ${K_m}$ of parallel weight $2$ that is new at level $\mathcal{N}_\ell$ and some prime $\lambda$ of $\q_\mathfrak{f}$ such that $\lambda\mid \ell$ and $\bar{\rho}_{F,\ell}\sim \bar{\rho}_{\mathfrak{f},\lambda}$. Since $\mathfrak{p} \nmid \ell\mathcal{N}_\ell$ and $\mathfrak{p} \>||\> \mathcal{N}$, by \Cref{rem:arises_consequence}, we have that 
	\begin{equation*}
		\pm(\Norm_{{K_m}/\q}(\mathfrak{p})+1)\equiv a_{\mathfrak{p}}(\mathfrak{f}) \mod \lambda.
	\end{equation*}
	This implies that 
	\begin{equation}\label{eqn:lnormdivide}
		\ell \mid \Norm_{\q_{\mathfrak{f}}/\q}(\Norm_{{K_m}/\q}(\mathfrak{p})+1 \pm a_{\mathfrak{p}}(\mathfrak{f})),
	\end{equation}
	Now, Theorem 0.1 in \cite{Livne} states that the Ramanujan-Petersson conjecture holds for a certain class of Hilbert modular forms. It is clear that the $\mathfrak{f}$ and $\mathfrak{p}$ we are considering satisfy the conditions of Theorem 0.1 in \cite{Livne}. Therefore, we have that
	$$ |a_{\mathfrak{p}}(\mathfrak{f})| \leq 2\sqrt{\Norm_{{K_m}/\q}(\mathfrak{p})} .$$ 
	Thus, \begin{align*}
		|\Norm_{{K_m}/\q}(\mathfrak{p})+1 \mp a_{\mathfrak{p}}(\mathfrak{f})|&
		\geq \Norm_{{K_m}/\q}(\mathfrak{p})+1-2\sqrt{\Norm_{{K_m}/\q}(\mathfrak{p})}\\
		&=\left(\sqrt{\Norm_{{K_m}/\q}(\mathfrak{p})}-1\right)^2>0.
	\end{align*}
	Hence, $\Norm_{\q_{\mathfrak{f}}/\q}(\Norm_{{K_m}/\q}(\mathfrak{p})+1 \mp a_{\mathfrak{p}}(\mathfrak{f}))\neq 0$. Therefore, by \eqref{eqn:lnormdivide}, $\ell$ is bounded by $\Norm_{\q_{\mathfrak{f}}/\q}(\Norm_{{K_m}/\q}(\mathfrak{p})+1 \mp a_{\mathfrak{p}}(\mathfrak{f}))$. Note that $\Norm_{{K_m}/\q}(\mathfrak{p})$ is either $p_0$ or $p_0^2$, and $p_0$ only depends on $B_1$ and $b$. Moreover, $\mathcal{N}_\ell$ has finitely many possible values that only depend on $2b$; more precisely, $\mathcal{N}_\ell$ is a finite product of prime ideals belonging to $\mathcal{S}$, each of which has a bounded exponent by \Cref{rem:boundNpowers23}. Thus, by noting that there are finitely many Hilbert eigenforms $\mathfrak{f}$ of any fixed level, we have that the set of Hilbert eigenforms $\mathfrak{f}$ with level equal to any $\mathcal{N}_{\ell}$ is finite. Therefore, for any $B_m$ an $\ell$th power with $\ell>\text{max}\{k,2b, C,p_0,5\}$,  $\Norm_{\q_{\mathfrak{f}}/\q}(\Norm_{{K_m}/\q}(\mathfrak{p})+1 \mp a_{\mathfrak{p}}(\mathfrak{f}))$ has finitely many possible values. By finding the maximum of these values, we obtain a bound on $\ell$.
\end{myproof}

\begin{rem}
	The condition that the point $P$ is non-integral was vital in our proof, since it allowed us to find the prime $p_0$ using \Cref{lem:primdivisorconstruct}. Also, if $b$ was negative, then to follow the same methodology we would require the modularity of elliptic curves over imaginary quadratic number fields, which is not known to be true yet. One might also be tempted to show that \Cref{thm: bPefectPowers} holds for the more general equation $y^2=x(x^2+ax+b)$, with some additional assumptions, by attempting to adjust this proof. However, this does not appear to be viable as we won't be able to control the common divisors of the two factors on the left hand side of \eqref{eqn:factor}, which was necessary for the effectiveness of the level-lowering method, i.e., obtaining a finite number of possible levels $\mathcal{N}_{\ell}$ when varying $\ell$.
\end{rem}

\end{document}